\documentclass[12pt]{amsart}
\usepackage{amssymb, amsmath, amsfonts, mathrsfs, dsfont}
\usepackage[utf8]{inputenx} 
\usepackage{mathtools}
\usepackage{hyperref}
\usepackage{enumerate}
\usepackage{tikz}
\usetikzlibrary{shapes.geometric}
\numberwithin{equation}{section}

\newtheorem{theorem}{Theorem}[section]
\newtheorem*{theorem*}{Theorem}
\newtheorem{lemma}[theorem]{Lemma}

\newtheorem*{corollary*}{Corollary}

\theoremstyle{definition}
\newtheorem{definition}[theorem]{Definition}
\newtheorem{remark}[theorem]{Remark}

\providecommand{\NN}{\ensuremath{\mathds{N}}} 
\providecommand{\RR}{\ensuremath{\mathds{R}}} 
\providecommand{\XX}{\ensuremath{\mathds{X}}} 
\providecommand{\ZZ}{\ensuremath{\mathds{Z}}} 
\providecommand{\sB}{\ensuremath{\mathscr{B}}} 
\providecommand{\sF}{\ensuremath{\mathscr{F}}} 
\providecommand{\sN}{\ensuremath{\mathscr{N}}} 
\providecommand{\sX}{\ensuremath{\mathscr{X}}} 
\providecommand{\cB}{\ensuremath{\mathcal{B}}} 
\providecommand{\cS}{\ensuremath{\mathcal{S}}} 
\providecommand{\cK}{\ensuremath{\mathcal{K}}} 
\providecommand{\cL}{\ensuremath{\mathcal{L}}} 
\providecommand{\beast}{\ensuremath{\mathcal{B}_{\infty}}} 
\providecommand{\bbH}{\ensuremath{\mathbb{H}}} 
\providecommand{\bbS}{\ensuremath{\mathbb{S}}} 
\providecommand{\dsP}{{\ensuremath{\mathds{P}}}} 
\providecommand{\dsE}{{\ensuremath{\mathds{E}}}} 
\providecommand{\conv}{\operatorname{conv}} 
\newcommand{\cech}{\operatorname{C}} 
\newcommand{\vietoris}{\operatorname{VR}} 
\newcommand{\dist}{\operatorname{d}} 
\newcommand{\real}{\operatorname{real}} 
\newcommand{\Haus}{\mathrm{Haus}}
\DeclarePairedDelimiterX{\abs}[1]{\lvert}{\rvert}{#1} 
\DeclarePairedDelimiterX{\norm}[1]{\lVert}{\rVert}{#1} 
\DeclarePairedDelimiterX{\skp}[2]{\langle}{\rangle}{#1, #2} 
\DeclarePairedDelimiterX{\sskp}[1]{\langle}{\rangle}{#1} 
\DeclarePairedDelimiterX{\cbras}[1]{\{}{\}}{{#1}} 
\DeclarePairedDelimiterX{\rbras}[1]{(}{)}{{#1}} 
\DeclarePairedDelimiterX{\sbras}[1]{[}{]}{{#1}}	
\DeclarePairedDelimiterX{\ceil}[1]{\lceil}{\rceil}{#1} 
\DeclarePairedDelimiterX{\floor}[1]{\lfloor}{\rfloor}{#1} 
\DeclarePairedDelimiterXPP{\Prob}[1]{\dsP}{(}{)}{}{#1} 
\DeclarePairedDelimiterXPP{\E}[1]{\dsE}{[}{]}{}{#1} 

\title{Gigantic random simplicial complexes}

\author[J.~Grygierek]{Jens Grygierek}
\address{Universit\"at Osnabr\"uck, Institut f\"ur Mathematik, 49069 Osnabr\"uck, Germany}
\email{jgrygierek@uni-osnabrueck.de}

\author[M.~Juhnke-Kubitzke]{Martina Juhnke-Kubitzke}
\email{juhnke-kubitzke@uni-osnabrueck.de}

\author[M.~Reitzner]{Matthias Reitzner}
\email{mreitzne@uni-osnabrueck.de}

\author[T.~R\"omer]{Tim R\"omer}
\email{troemer@uni-osnabrueck.de}

\author[O.~R\"ondigs]{Oliver R\"ondigs}
\email{oliver.roendigs@uni-osnabrueck.de}
\begin{document}

\begin{abstract}
  We provide a random simplicial complex by applying standard
  constructions to a Poisson point process in Euclidean space.
  It is gigantic in the sense that~--~up to homotopy equivalence~--~it almost surely
  contains infinitely many copies of every
  compact topological manifold, 
  both in isolation and in percolation.
\end{abstract}


\maketitle

%
%
%
%

\section{Introduction}
\label{sec:intr-main-results}

A debate on prominent and interesting topological spaces 
would
most likely mention topological manifolds, that is,
topological Hausdorff spaces which are locally homeomorphic
to a Euclidean space. 
More subtle is the question of randomly choosing topological manifolds,
a task subject to active research in this century.
The restriction
to special types of manifolds -- like Riemann surfaces
\cite{brooks-makover}, 3-manifolds \cite{dunfield-thurston}, or
configuration spaces \cite{farber-kappeler} -- 
indicates its subtlety. 
Even collections of submanifolds of a given topological manifold
tend to be very unwieldy, at least through the eyes of
probability theory. This article presents a combinatorial
solution, that is, a naturally constructed
random simplicial complex, which is rich
enough to realize any compact topological manifold, at least up
to homotopy equivalence.

Recall that a result of Milnor 
\cite[Theorem 1]{milnor.cw}
states that any topological manifold is homotopy equivalent
to a countable locally finite simplicial complex. 
Examples of noncompact topological manifolds
like a connected sum of infinitely many tori show
that finiteness cannot be expected in general.
However, a refinement 
due to Kirby and Siebenmann \cite{kirby-siebenmann.hauptvermutung},
\cite{kirby-siebenmann.essays} states 
that any compact topological manifold is homotopy equivalent
to a finite simplicial complex. Hence from the viewpoint of 
homotopy theory, finite simplicial complexes constitute a
rich and highly interesting class of topological spaces.
Note furthermore that Manolescu, based on work
of Galewski-Stern \cite{GalewskiStern} (among others), disproved the
Triangulation Conjecture: For every integer $d\geq 5$, there
exists a compact topological manifold of dimension $d$ which is 
not homeomorphic to a finite simplicial 
complex \cite{manolescu}. This perhaps unexpected
result may be viewed as an argument
towards considering topological manifolds up to homotopy
equivalence. Up to dimension three, every topological manifold
is homeomorphic to a simplicial complex \cite{moise}. 
In the notoriously exotic
dimension four, the Casson invariant proves that  
Freedman's $E_8$ manifold \cite{freedman} 
is a compact topological manifold not
homeomorphic to a simplicial complex, as explained for
example in~\cite{akbulut-mccarthy}. 

The aforementioned debate thus should also list 
finite simplicial complexes as examples of
interesting topological spaces. 
Our work provides a natural construction of a random simplicial complex
associated with Euclidean space $\RR^d$ which ``contains all of them''
in the following two different ways:
\begin{description}
\item[Lonely Complex] 
  Almost surely, it contains infinitely many copies of any given 
  $d$-embeddable
  finite simplicial complex as isolated components.
\item[Giant Beast] Almost surely, it
  contains an unbounded connected component having infinitely many
  copies of any 
  given $d$-embeddable finite simplicial complex as wedge summands.
\end{description}
Here a simplicial complex $\cK$ is called $d$-{\em embeddable\/} 
if there exists a piecewise linear embedding 
$\cK\to \RR^d$. Recall that a finite simplicial complex of
dimension $d$ is at least $(2d+1)$-embeddable. Recall also that
the wedge sum of two pointed simplicial complexes (or general pointed
topological spaces) is the one-point union at the respective 
basepoints. Up to homotopy it is the same as a ``whiskered
one-point union'': joining the two basepoints
via an additional edge.
Before giving both
the construction and a precisely formulated theorem, a consequence on homology
groups is readily obtained. 

\begin{theorem*}
For every $n\in \NN$  there exists a random simplicial complex whose homology 
almost surely contains the homology of any finite simplicial complex of dimension at most $n$ infinitely many times as a direct summand.
\end{theorem*}

This theorem is weaker than the statements on the lonely complex and the giant beast in the sense that homology is determined
by homotopy type, but not vice versa.
Now to its 
construction, coming in two flavours: 
Any subset -- which in our case is locally finite -- of $\RR^d$ 
gives rise to at least two types of abstract simplicial complexes, a 
{\em Vietoris-Rips\/} complex and a {\em \v{C}ech\/} complex, depending  on the ambient Euclidean metric and a real positive parameter $\rho$. The subset itself is a random one, given by a stationary Poisson point process $\eta$ in $\RR^d$ with intensity measure $t\cdot \Lambda_d$,
where $0<t\in \RR$ and $\Lambda_d$ refers to the Lebesgue measure.
Precise definitions
are given in Section~\ref{sec:simplicial-complexes}, namely
Definitions~\ref{def:vr},~\ref{def:cc}, and 
Subsection~\ref{sec:poiss-point-proc}. 
The resulting theorems 
on random Vietoris-Rips and \v{C}ech complexes, respectively, 
also come in two flavours: isolation and percolation. For the
sake of brevity, only the \v{C}ech version will be stated in
this introduction; see Theorems \ref{thm:iso_vr} and~\ref{thm:perc_vr}
for the statements on the random Vietoris-Rips complex.

\begin{theorem*}[Lonely complexes]
Let $\cL$ be a $d$-embeddable simplicial complex. Then for  every stationary Poisson point process on $\RR^d$, 
any random \v{C}ech complex 
contains almost surely infinitely many 
  isolated simplicial complexes that are homotopy equivalent to $\cL$.
\end{theorem*}

\begin{theorem*}[The giant beast]
 Let $\cK$ be a $d$-embeddable simplicial complex. 
 Assume that the  random \v{C}ech complex of a stationary Poisson point process on $\RR^d$
  contains an unbounded connected simplicial complex $\cB_\infty$.  
Then the random  \v{C}ech complex  contains almost surely infinitely many subcomplexes 
that are homotopy equivalent to $\cK$ and only connected to $\beast$ by an edge.
\end{theorem*}

Observe that by the groundbreaking work of Meester and Roy \cite{meester-roy-art, meester-roy-book} the question whether there is an unbounded connected complex and in particular its uniqueness is well understood.
For each $\rho>0$ there exists $t_{\text{perc}}(\rho)>0$, 
such that for every stationary Poisson point process $\eta$ on $\RR^d$ 
with intensity $t > t_{\text{perc}}(\rho)$ the random \v{C}ech complex $\cech(\eta, \rho)$ 
contains one unbounded connected simplicial complex $\beast$, see also Theorem \ref{th:percolation}.

It is worth noting that every finite  simplicial
complex $\cK$ admits
a geometric realization $\real(\cK)$ in  Euclidean space and 
a real number $\rho>0$ such that the \v{C}ech complex associated with
a suitable finite subset of $\real(\cK)$ and $\rho$ is
homotopy equivalent to $\cK$. Hence, for a suitable $d$, the preceding results apply to 
any finite simplicial complex, up to homotopy equivalence. 

One distinguishing feature of our model, in comparison with the extensive
list of random simplicial complexes given for example in 
\cite{linial-meshulam}, \cite{meshulam-wallach} or \cite{kahle.clique}
(see Kahle's chapter in \cite{handbook-discrete} for an overview and further references,
also to other survey articles), is its infinite size. 
Most of the previous investigations and results concentrated on finite random simplicial complexes
which may be obtained by restricting to some bounded subset in $\RR^d$. 
Important results in this direction which we want to emphasize in our context are results on Betti numbers of random simplicial complexes due to 
Kahle \cite{Kahle2011}, Kahle and Meckes (proving limit theorems) \cite{KahleMeckes2013, KahleMeckes2016}, Decreusefond et.al. \cite{Decreusefondetal2011}, and recently Adler, Subag and Yogeshwaran \cite{Yogeshwaran-Subag-Adler}, see also \cite{Yogeshwaran-Adler} for results on more general point processes. These results should be compared with the following immediate consequence of our main results.

\begin{corollary*}
Almost surely any list of $d$ natural numbers is
realized infinitely many times as  Betti numbers of isolated subcomplexes, and if a giant component 
exists also as Betti numbers of wedge summand 
subcomplexes of our random simplicial complex. 
\end{corollary*}

The research on geometric random simplicial complexes goes back to work of Gilbert \cite{Gil} who 
introduced the random geometric graph which is in the background of the construction of the 
\v{C}ech, resp. Vietoris-Rips complex. A thorough treatment of this random geometric graph was given 
in the seminal book of Penrose \cite{PenroseRandomGeometricGraphs}  where subgraph counts are at the 
core of the investigations. Concentration inequalities for subgraph counts have very recently been 
obtained by Bachmann \cite{Bachmann} and Bachmann and Reitzner \cite{BachmannReitzner}, and it would 
be interesting if these concentration results can be extended to random Betti numbers.

In recent years there have been prominent activities on topological data analysis and persistent homology, where
reconstrutions of topological structure from data 
sets \cite{elz}, \cite{edelsbrunner-harer}, \cite{zomorodian-carlsson} are discussed. Roughly
stated, one viewpoint here is the following: The data, 
a point cloud in Euclidean space, stems from
sampling a suitable topological manifold. A natural question is
to detect topological features which may occur in a totally
random point cloud in Euclidean space. The richness of the 
random point cloud presented in this paper could indicate that it is rather the absence than
the presence of certain topological properties which characterizes
non-random data.

The article is structured as follows. Section~\ref{sec:simplicial-complexes}
provides information on simplicial complexes relevant to the present discussion,
and Section~\ref{sec:poiss-point-proc} achieves the same for Poisson point 
processes. The final section~\ref{sec:rand-simpl-compl} contains proofs of
slightly stronger versions of the theorems mentioned in this introduction.
We tried to present this work such that it may be accessible to both 
stochastic geometers and algebraic topologists.

\section{Simplicial complexes}
\label{sec:simplicial-complexes}
%
%

Simplicial complexes have various incarnations. In the sequel, we will consider both abstract simplicial complexes and geometric simplicial complexes, as described for example in \cite[\S 2, \S 3]{munkres}. Any geometric simplicial complex can be viewed as
an abstract simplicial complex by using its vertex scheme. 
Conversely, any abstract simplicial complex $\cK$
admits a geometric realization $\real(\cK)$. We note that any two geometric realizations of $\cK$ are homeomorphic. 
In what follows, the term ``simplicial complex'' may be used without
further adjectives; the reader should be able to specify that
from the context. The types of constructions employed here produce
abstract simplicial complexes with points of vertices being
embedded in Euclidean space.

\begin{definition}
  Two simplicial complexes  $\cK$ and $\cL$ on vertex sets $V(\cK)$ and $V(\cL)$, respectively, are {\em combinatorially equivalent\/} 
  if there exists a bijection $\varphi\colon V(\cK) \rightarrow V(\cL)$ 
  such that $F \in \cK$ if and only if $\varphi(F) \in \cL$.
\end{definition}

Every finite simplicial complex is combinatorially equivalent to a subcomplex of
a standard $N$-simplex $\Delta^N$ \cite[Corollary 2.9]{munkres}.
Examples of simplicial complexes arise naturally from metric spaces. Two such instances are provided by Vietoris-Rips and \v{C}ech complexes.

\begin{definition}[Vietoris-Rips complex]\label{def:vr}
  Let $X=(X,\mathbf{d})$ be a metric space 
  (usually a locally finite subset of $\RR^d$) and $0<\rho\in\RR$.
  The {\em Vietoris-Rips complex\/} of $(X,\mathbf{d})$ with respect to $\rho$
  \[ \vietoris(X,\mathbf{d},\rho) = \vietoris(X,\rho) \]
  is the abstract simplicial complex on vertex set $X$ whose $k$-simplices are all subsets
  $\{x_0,\dotsc,x_k\}\subset X$ with $\mathbf{d}(x_i,x_j) \leq \rho$ for all
  $0\leq i,j\leq k$.
\end{definition}

The Vietoris-Rips complex of $X$ (w.r.t. $\rho$) is determined by its $1$-skeleton
in the sense that a subset $\{x_0,\dotsc,x_k\}\subset X$ is
a $k$-simplex if and only if all its subsets $\{x_i,x_j\}$ with
two elements are $1$-simplices. In other words, the Vietoris-Rips complex equals the clique complex of the graph on vertex set $X$ whose  edges are those pairs of vertices with distance smaller than or equal to $\rho$. 

\begin{definition}[\v{C}ech complex]\label{def:cc}
  Let $X$ be a subset of a metric space $(Y,\mathbf{d})$ 
  (usually a locally finite subset of $\RR^d$) and $0<\rho\in\RR$.
  The {\em \v{C}ech complex\/} of $X\subset(Y,\mathbf{d})$ with respect to $\rho$
  \[ \cech(X\subset Y,\mathbf{d},\rho) = \cech(X,\rho) \]
  is the abstract simplicial complex whose $k$-simplices are all subsets
  $\{x_0,\dotsc,x_k\}\subset X$ admitting a point $y\in Y$ with 
  $\mathbf{d}(x_i,y) \leq \tfrac{\rho}{2}$ for all
  $0\leq i\leq k$.
\end{definition}

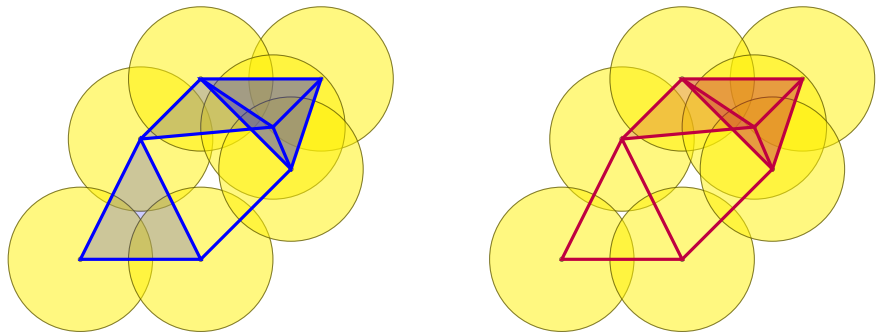
\begin{figure}
\begin{center}
  \begin{tikzpicture}[scale=0.8]
    {\draw[fill,color=yellow,fill opacity=0.5,draw opacity=0.5] (-1,2) circle (1.2cm);}
    {\draw[draw opacity=0.5] (-1,2) circle (1.2cm);}
    {\draw[draw opacity=0.5] (-1,2) circle (1pt);}

    {\draw[fill,color=yellow,fill opacity=0.5,draw opacity=0.5] (-2,0) circle (1.2cm);}
    {\draw[draw opacity=0.5] (-2,0) circle (1.2cm);}
    {\draw[draw opacity=0.5] (-2,0) circle (1pt);}

    {\draw[fill,color=yellow,fill opacity=0.5,draw opacity=0.5] (0,0) circle (1.2cm);}
    {\draw[draw opacity=0.5] (0,0) circle (1.2cm);}
    {\draw[draw opacity=0.5] (0,0) circle (1pt);}

    {\draw[fill,color=yellow,fill opacity=0.5,draw opacity=0.5] (2,3) circle (1.2cm);}
    {\draw[draw opacity=0.5] (2,3) circle (1.2cm);}
    {\draw[draw opacity=0.5] (2,3) circle (1pt);}

    {\draw[fill,color=yellow,fill opacity=0.5,draw opacity=0.5] (0,3) circle (1.2cm);}
    {\draw[draw opacity=0.5] (0,3) circle (1.2cm);}
    {\draw[draw opacity=0.5] (0,3) circle (1pt);}

    {\draw[fill,color=yellow,fill opacity=0.5,draw opacity=0.5] (1.2,2.2) circle (1.2cm);}
    {\draw[draw opacity=0.5] (1.2,2.2) circle (1.2cm);}
    {\draw[draw opacity=0.5] (1.2,2.2) circle (1pt);}

    {\draw[fill,color=yellow,fill opacity=0.5,draw opacity=0.5] (1.5,1.5) circle (1.2cm);}
    {\draw[draw opacity=0.5] (1.5,1.5) circle (1.2cm);}
    {\draw[draw opacity=0.5] (1.5,1.5) circle (1pt);}

    {\draw[very thick,color=blue] (-1,2) -- (-2,0) -- (0,0) -- cycle;}
    {\draw[very thick,color=blue] (1.5,1.5) -- (0,3) -- (2,3) -- cycle;}
    {\draw[very thick,color=blue] (0,0) -- (1.5,1.5);}
    {\draw[very thick,color=blue] (-1,2) -- (0,3);}
    {\draw[very thick,color=blue] (1.2,2.2) -- (-1,2);}
    {\draw[very thick,color=blue] (1.2,2.2) -- (2,3);}
    {\draw[very thick,color=blue] (1.2,2.2) -- (0,3);}
    {\draw[very thick,color=blue] (1.2,2.2) -- (1.5,1.5);}

    {\draw[fill,color=blue,fill opacity=0.2] (-1,2) -- (-2,0) -- (0,0) -- cycle;}
    {\draw[fill,color=blue,fill opacity=0.2] (1.5,1.5) -- (0,3) -- (2,3) -- cycle;}
    {\draw[fill,color=blue,fill opacity=0.2] (-1,2) -- (0,3) -- (1.2,2.2) -- cycle;}
    {\draw[fill,color=blue,fill opacity=0.2] (1.5,1.5) -- (0,3) -- (1.2,2.2) -- cycle;}
    {\draw[fill,color=blue,fill opacity=0.2] (2,3) -- (0,3) -- (1.2,2.2) -- cycle;}
    {\draw[fill,color=blue,fill opacity=0.2] (1.5,1.5) -- (2,3) -- (1.2,2.2) -- cycle;}

    {\draw[xshift=8cm,fill,color=yellow,fill opacity=0.5,draw opacity=0.5] (-1,2) circle (1.2cm);}
    {\draw[xshift=8cm,draw opacity=0.5] (-1,2) circle (1.2cm);}
    {\draw[xshift=8cm,draw opacity=0.5] (-1,2) circle (1pt);}

    {\draw[xshift=8cm,fill,color=yellow,fill opacity=0.5,draw opacity=0.5] (-2,0) circle (1.2cm);}
    {\draw[xshift=8cm,draw opacity=0.5] (-2,0) circle (1.2cm);}
    {\draw[xshift=8cm,draw opacity=0.5] (-2,0) circle (1pt);}

    {\draw[xshift=8cm,fill,color=yellow,fill opacity=0.5,draw opacity=0.5] (0,0) circle (1.2cm);}
    {\draw[xshift=8cm,draw opacity=0.5] (0,0) circle (1.2cm);}
    {\draw[xshift=8cm,draw opacity=0.5] (0,0) circle (1pt);}

    {\draw[xshift=8cm,fill,color=yellow,fill opacity=0.5,draw opacity=0.5] (2,3) circle (1.2cm);}
    {\draw[xshift=8cm,draw opacity=0.5] (2,3) circle (1.2cm);}
    {\draw[xshift=8cm,draw opacity=0.5] (2,3) circle (1pt);}

    {\draw[xshift=8cm,fill,color=yellow,fill opacity=0.5,draw opacity=0.5] (0,3) circle (1.2cm);}
    {\draw[xshift=8cm,draw opacity=0.5] (0,3) circle (1.2cm);}
    {\draw[xshift=8cm,draw opacity=0.5] (0,3) circle (1pt);}

    {\draw[xshift=8cm,fill,color=yellow,fill opacity=0.5,draw opacity=0.5] (1.2,2.2) circle (1.2cm);}
    {\draw[xshift=8cm,draw opacity=0.5] (1.2,2.2) circle (1.2cm);}
    {\draw[xshift=8cm,draw opacity=0.5] (1.2,2.2) circle (1pt);}

    {\draw[xshift=8cm,fill,color=yellow,fill opacity=0.5,draw opacity=0.5] (1.5,1.5) circle (1.2cm);}
    {\draw[xshift=8cm,draw opacity=0.5] (1.5,1.5) circle (1.2cm);}
    {\draw[xshift=8cm,draw opacity=0.5] (1.5,1.5) circle (1pt);}

    {\draw[xshift=8cm,very thick,color=purple] (-1,2) -- (-2,0) -- (0,0) -- cycle;}
    {\draw[xshift=8cm,very thick,color=purple] (1.5,1.5) -- (0,3) -- (2,3) -- cycle;}
    {\draw[xshift=8cm,very thick,color=purple] (0,0) -- (1.5,1.5);}
    {\draw[xshift=8cm,very thick,color=purple] (-1,2) -- (0,3);}
    {\draw[xshift=8cm,very thick,color=purple] (1.2,2.2) -- (-1,2);}
    {\draw[xshift=8cm,very thick,color=purple] (1.2,2.2) -- (2,3);}
    {\draw[xshift=8cm,very thick,color=purple] (1.2,2.2) -- (0,3);}
    {\draw[xshift=8cm,very thick,color=purple] (1.2,2.2) -- (1.5,1.5);}
    {\draw[xshift=8cm,fill,color=purple,fill opacity=0.2] (1.5,1.5) -- (0,3) -- (2,3) -- cycle;}
    {\draw[xshift=8cm,fill,color=purple,fill opacity=0.2] (-1,2) -- (0,3) -- (1.2,2.2) -- cycle;}
    {\draw[xshift=8cm,fill,color=purple,fill opacity=0.2] (1.5,1.5) -- (0,3) -- (1.2,2.2) -- cycle;}
    {\draw[xshift=8cm,fill,color=purple,fill opacity=0.2] (2,3) -- (0,3) -- (1.2,2.2) -- cycle;}
    {\draw[xshift=8cm,fill,color=purple,fill opacity=0.2] (1.5,1.5) -- (2,3) -- (1.2,2.2) -- cycle;}
  \end{tikzpicture}
\end{center}
\caption{Vietoris-Rips and \v{C}ech complex of the same 7-element set}
\end{figure}
It directly follows from the definitions that the \v{C}ech complex is a subcomplex of the Vietoris-Rips w.r.t. the same parameter $\rho$ and that their $1$-skeleta coincide. Moreover, it is well-known that the Vietoris-Rips complexes can be squeezed between two \v{C}ech complexes in the following way: 
\begin{lemma}\label{lem:vr-1skeleton-cech}
  Let $X$ be a subset of $\RR^d$ and $0<\rho\in \RR$. Then
  \[
  \cech(X,\rho)\subseteq \vietoris(X,\rho)\subseteq\cech(X,2\rho).
  \]
  Moreover, $\cech(C,\rho)$ and $\vietoris(X,\rho)$ have the same $1$-skeleton.
\end{lemma}

Let $0<\rho\in \RR$. We say that a finite simplicial complex $\cK$  admits a  \emph{($d$-dimensional) \v{C}ech representation  with respect to\/} $\rho$ if there exists a finite subset $X\subset \RR^d$ such that
$\cK$ is combinatorially equivalent to $\cech(X,\rho)$. In this case, we call $X$ a \emph{($d$-dimensional) \v{C}ech representation} of $\cK$.  
The definition of a \emph{Vietoris-Rips representation} is analogous.

\begin{remark}
Given those definitions, one of the first questions one might ask is, which simplicial complexes $\cK$ do admit a \v{C}ech or a Vietoris-Rips representation. In the latter case, $\cK$ necessarily has to be a \emph{flag} complex, i.e., a simplicial complex whose inclusion-minimal missing faces are of cardinality $2$. But besides this obvious condition, no other results yielding an answer to the posed question exist. 
 It is  therefore natural to relax the original question and to only ask for such a classification up to homotopy equivalence. 
In fact, for \v{C}ech representations a complete characterization is known: On the one hand, if $\phi\colon \cK\to \RR^d$ is a piecewise linear embedding of a simplicial complex, then it follows from the Nerve Lemma (see e.g., \cite[Theorem 10.6]{bj}) that the \v{C}ech complex of a sufficiently dense point set in $\phi(\cK)$ with respect to a small enough distance parameter $\rho$ is homotopy equivalent to $\real(\cK)=\phi(\cK)$. On the other hand, another application of the Nerve Lemma shows that any \v{C}ech complex of a finite point set in $\RR^d$ is homotopy equivalent to a $d$-embeddable 
finite simplicial complex. Hence, \v{C}ech complexes of finite point sets in $\RR^d$ model precisely the  homotopy types of $d$-embeddable finite simplicial complexes. In particular, the homology groups and the Betti numbers of such vanish in degrees above $d-1$.

For Vietoris-Rips complexes, the situation is more complicated. On the one hand, it is shown in \cite[Theorem II]{AdamszekFrickVakiil} that every $d$-embeddable simplicial complex $\cK$ admits a $d$-dimensional  Vietoris-Rips representation up to homotopy equivalence, i.e., there exists a point set $X\subseteq \RR^d$ and a distance parameter $\rho$ such that $\vietoris(X,\rho)$ has the same homotopy type as $\cK$. On the other hand, it is easy to see that the boundary complex of a cross-polytope of dimension $n$ admits even a $2$-dimensional Vietoris-Rips presentation. Indeed, the Vietoris-Rips complex of $2n$ points that are equidistantly distributed on the unit sphere with respect to parameter $\rho=1-\epsilon$ (for $\epsilon$ small enough) is combinatorially equivalent to the boundary complex of the $n$-dimensional cross-polytope. Hence, any wedge sum of spheres of arbitrary dimensions occurs as the homotopy type of a  simplicial complex admitting a $d$-dimensional Vietoris-Rips representation. 
Considering homology and topological Betti numbers instead of homotopy equivalence, this also means that every list of positive integers can occur as the Betti numbers of a simplicial complex with a $d$-dimensional Vietoris-Rips representation, in contrast to the situation for \v{C}ech complexes. 
\end{remark}

Given one \v{C}ech representation $X\subset \RR^d$
of $\cK$ with respect to $\rho$, it is natural to
ask whether $X$ may be perturbed slightly (while staying a \v{C}ech representation). This question 
will be answered for both \v{C}ech and Vietoris-Rips complexes in Theorem~\ref{thm:all_is_cr}
after introducing two auxiliary notions.

\begin{definition}\label{def:distance}
  Let $X,Y \subset \RR^d$ be finite nonempty subsets, with cardinalities
  $\abs{X}$ and $\abs{Y}$, respectively.
  If $|X|=|Y|$, then the \emph{distance} between $X$ and $Y$ is the maximum distance of the pairs of points $(x,y)\in X\times Y$ with respect to the optimal 
  bijective map $f\colon X \rightarrow Y$. Otherwise, $X$ and $Y$ are said to be of distance $\infty$:
  \begin{align*}	
    \dist(X,Y) = \begin{cases}
      \min\limits_{\substack{f\colon X \rightarrow Y \\ f \text{ bijective}}} \max\limits_{x \in X} \norm{f(x) - x}_2, & \quad \text{if } \abs{X} = \abs{Y},\\
      \infty, & \quad \text{if } \abs{X} \neq \abs{Y}.
    \end{cases}
  \end{align*}
\end{definition}

This notion of distance for finite subsets of $\RR^d$ 
is well suited for our purposes, and may be compared
with the classical Hausdorff distance 
\[ d_\Haus(X,Y) = \max\{\sup_{x\in X}\inf_{y\in Y}\norm{x-y}_2,\sup_{y\in Y}\inf_{x\in X}\norm{x-y}_2\} \]
of subsets $X,Y\subset \RR^d$ as follows. For finite subsets $X,Y\subset \RR^d$
one always has an inequality
$d_\Haus(X,Y)\leq \dist(X,Y)$, and equality holds for subsets with
$1\leq \abs{X}=\abs{Y}\leq 2$. The triangles $X=\{(0,0),(1,0),(1,1)\}$ and
$Y=\{(-1,1),(0,1),(-1,0)\}$ in $\RR^2$ satisfy
$d_\Haus(X,Y) = \sqrt{2}<2=\dist(X,Y)$.

\begin{definition}\label{def:generic-cech-vietoris}
  Let $0<\rho\in \RR$ and let $\cK$ be a finite simplicial complex. A $d$-dimensional \v{C}ech representation $X\subset \RR^d$ of $\cK$ 
  with respect to $\rho$ is called {\em generic\/}
  if there exists a parameter $\delta > 0$ 
  such that for every $Y\subset \RR^d$ with $\dist(X,Y) < \delta$, the 
  simplicial complex $\cech(Y, \rho)$ 
  is combinatorially equivalent 
  to $\cK$. A generic $d$-dimensional Vietoris-Rips representation is
  defined analogously.
\end{definition}

\begin{remark}
  By Definition~\ref{def:generic-cech-vietoris}, every vertex of a 
  generic \v{C}ech representation $X\subset \RR^d$ may be moved in a
  small neighborhood, in particular a $\delta$-ball (for $\delta$ sufficiently small), without affecting the combinatorial properties 
  of the resulting simplicial complex. 
  This allows us to reduce the occurrence of a specific simplicial complex in our randomized model to the event that, in a sufficiently large window, every point of our Poisson point process lies in a $\delta$-ball around a vertex of a given simplicial complex and every such ball contains exactly one point of our Poisson point process.
\end{remark}

\begin{theorem}\label{thm:all_is_cr}
  Let $0<\rho\in \RR$. If a 
  finite simplicial complex $\cK$ admits a $d$-dimensional \v{C}ech representation
  with respect to $\rho$, then 
  it admits a generic $d$-dimensional \v{C}ech representation with
  respect to $\rho$. The analogous statement is true for
  Vietoris-Rips representations.
\end{theorem}

\begin{proof}
    We show the claim for \v{C}ech complexes. Lemma~\ref{lem:vr-1skeleton-cech} implies that the claim for Vietoris-Rips complexes directly follows from the one for \v{C}ech complexes.
    
    Let $X \subset \RR^d$ be a $d$-dimensional \v{C}ech representation of $\cK$ with respect to $\rho$, i.e., $\cK$ is combinatorially equivalent to $\cech(X, \rho)$. We show that there exists a \v{C}ech representation $Y \subset \RR^d$ and $\delta > 0$ such that any points set $Z \subset \RR^d$ with $\dist(Y,Z) < \delta$ is a \v{C}ech representation of $\cK$.
    
    For a finite point set $A \subset \RR^d$ we denote by $\bbS_A$ its minimal bounding sphere, i.e., the $d$-dimensional sphere with minimal radius that contains $A$. Let further $r_A$ respectively $m_A$  denote the radius respectively the center of $\bbS_A$.
    It is straightforward to show that for $A \subseteq X$ we have $A \in \cech(X, \rho)$ if and only if $r_A \leq \rho$.
    We distinguish two cases.	
    
    \hspace*{1em} \textsf{Case 1:} $r_A< \rho$ for all $A\in \cech(X, \rho)$.
    
    Let
\begin{align*}
      \delta := \min\cbras*{|r_A-\rho|~:~A\subseteq X}.
\end{align*}
    Note that, by assumption, we have $\delta > 0$ and the minimum is attained for a face $A$ of $\cech(X,\rho)$. Let $Z \subseteq \RR^d$ be a point set with $\dist(X,Z) \leq \frac{\delta}{2}$ and let $f:X \rightarrow Z$ be a bijection such that $\dist(X,Z) = \max\limits_{x\in X}\norm{f(x)-x}_2$. We show that $\cech(X, \rho)$ is combinatorially equivalent to $\cech(Z, \rho)$ via the map $f$, 
    i.e., $A \in \cech(X, \rho)$ if and only if $f(A) \in \cech(Z, \rho)$.
    If $A \in \cech(X, \rho)$, then $r_A \leq {\rho-\delta}$. Since 
    \begin{align*}
      \norm{f(x) - m_A} \leq \norm{x - m_A} + \frac{\delta}{2} \leq r_A + \frac{\delta}{2} \leq \rho - \frac{\delta}{2} < \rho			
    \end{align*}
    for all $x \in A$, we have $r_{f(A)} \leq \rho$ and hence $f(A) \in \cech(Z, \rho)$. 
    If $A \notin \cech(X, \rho)$, then $r_A \geq \rho + \delta$. This implies
    \begin{align*}
      r_{f(A)} \geq r_A - \frac{\delta}{2} \geq \rho + \frac{\delta}{2} >\rho
    \end{align*}
    and therefore $f(A) \notin \cech(Z, \rho)$. The claim follows. In particular, we have shown that the \v{C}ech representation $X$ is already generic.
    
    \hspace*{1em} \textsf{Case 2:} There exists $A\in \cech(X,\rho)$ such that $r_A=\rho$.
    
    Let
    \begin{align*}
      \delta := \min\cbras*{r_A-\rho~ :~A \notin \cech(X, \rho)}.
    \end{align*}		 
    As $r_A > \rho$ for $A \not\in \cech(X, \rho)$, we have $\delta > 0$.
    Let 
		$$
		Y := \left\{\frac{\rho}{\rho+\tfrac{9}{10}\delta}\cdot x~:~x\in X\right\}\subseteq \RR^d$$
		be a ``scaled version'' of $X$. 
    We claim that $\cech(X, \rho)$ is combinatorially equivalent to $\cech(Y, \rho)$ via the natural bijection $f: X \rightarrow Y: x \mapsto \frac{\rho}{\rho + \tfrac{9}{10}\delta} \cdot x$. 
    Let $A \in \cech(X, \rho)$, i.e., $r_A \leq \rho$. 
    As $Y$ is obtained from $X$ by scaling it with $\frac{\rho}{\rho+\tfrac{9}{10}\delta}<1$, we conclude
    \begin{align*}
      r_{f(A)} = r_A \cdot \frac{\rho}{\rho+\tfrac{9}{10}\delta} < r_A\leq \rho		
    \end{align*}
    and hence $f(A) \in \cech(Y, \rho)$. 
    If $A \notin \cech(X, \rho)$, then $r_A \geq \rho+\delta$ and thus
    \begin{align*}
      r_{f(A)} = \frac{\rho}{\rho+\tfrac{9}{10}\delta} \cdot r_A \geq \frac{\rho}{\rho+\tfrac{9}{10}\delta} \cdot (\rho+\delta) > \rho,
    \end{align*}
    which implies $f(A) \notin \cech(Y, \rho)$. 
    It now follows from Case 1 that $Y$ is a generic \v{C}ech representation of $\cK$.
  \end{proof}

 %

%
\section{Poisson Point Processes}
\label{sec:poiss-point-proc}

Let $(\XX, \sX)$ be a measurable space and $(\Omega, \sF, \dsP)$ a fixed underlying probability space.

\begin{definition}
	We denote by $N_\sigma := N_\sigma(\XX)$ the space of all $\sigma$-finite measures $\chi$ on $\XX$, with $\chi(B) \in \NN_0 \cup \{ \infty \}$ for all $B \in \sX$, and by $\sN_\sigma := \sN_\sigma(\XX)$ the smallest $\sigma$-algebra on the set $N_\sigma$, such that the mappings $\chi \rightarrow \chi(B)$ are measurable for all $B \in \sX$.
\end{definition}

\begin{definition}
A point process is a measurable mapping from our underling probability space $(\Omega, \sF, \dsP)$ to the space of all counting measures $(N_\sigma, \sN_\sigma)$, i.e. for all $B \in \sX$ and all $k \in \NN_0$ it holds that
	\begin{align*}
		\left\{ \omega \in \Omega : \eta(\omega)(B) = k \right\} \in \sF.
	\end{align*}
\end{definition}

To shorten our notation we write $\eta(B)$ for the random variable $\eta(\omega)(B)$.
Thus a  point process $\eta$ is a discrete measure having mass concentrated at random points in the underlying space $\XX$.
	To simplify our Notation we will often handle $\eta$ as a random set of points given by
	\begin{align*}
		x \in \eta \Leftrightarrow x \in \cbras*{y \in \XX : \eta(\cbras{y}) > 0}.
	\end{align*}

	The intensity measure $\mu$ of an point process $\eta$ on the space $(\XX, \sX)$ is defined as the expected number of points of $\eta$ laying in the set $B$. Hence 
	\begin{align*}
		\mu(B) := \E{\eta(B)}, \quad B \in \sX.
	\end{align*}

\begin{definition}
	Consider a $\sigma$-finite measure $\mu$ on $\XX$. Then a Poisson point process $\eta$ with intensity measure $\mu$ on $\XX$ satisfies the following properties:
	\begin{enumerate}
		\item For all $B \in \sX$ and all $k \in \NN_0$ it holds, that $\eta(B) \overset{d}{\sim} \text{Po}_{\mu(B)}$, i.e.,
			\begin{align*}
				\Prob*{\eta(B) = k} = \frac{\mu(B)^k}{k!}e^{-\mu(B)},
			\end{align*}
			and for $\mu(B) = \infty$, we set $\frac{\infty^k}{k!}e^{-\infty} = 0$ for all $k$.
		\item For all $m \in \NN$ and all pairwise disjoint measurable sets $B_1, \ldots, B_m \in \sX$, the random variables $\eta(B_1), \ldots, \eta(B_m)$ are independent.
	\end{enumerate}
\end{definition}

In the case $\XX = \RR^d$ in which we are interested in this paper, a point process is called stationary if its intensity measure $\mu$ is invariant under translations. This implies that $\mu = t \cdot \Lambda_d$, where $t>0$ and $\Lambda_d$ denotes the $d$-dimensional Lebesgue measure.

	\begin{remark}[Cube construction]
\label{rem:cube-construction}	
We consider the space $(\RR^d, \sB^d)$, where $\sB^d$ denotes the Borel-$\sigma$-algebra on $\RR^d$ and divide the space $\RR^d$ into countable many cubes along the lattice $\ZZ^d$, i.e.
	\begin{align*}
		\RR^d = \bigcup\limits_{v \in \ZZ^d} \rbras*{v + [0,1]^d}
	\end{align*}
	where we used $v + [0,1]^d$ as short notation for the set $\cbras{v + x \in \RR^d : x \in [0,1]^d}$.
	 This allows us to construct a stationary Poisson point process $\eta$ on $\RR^d$ with intensity measure $\mu := t \Lambda_d$ where $t > 0$ and $\Lambda_d$ denotes the $d$-dimensional Lebesgue measure. 
	 We are taking a sequence $(P_v)_{v \in \ZZ^d}$ of independent Poisson distributed random variables with mean value $\mu([0,1]) = t$ and a double indexed sequence $(X_{(v,j)})_{(v,j) \in \ZZ^d \times \NN}$ of independent uniformly distributed random variables on the unit cube $[0,1]^d$. Note that we assume independence of all random variables used, especially we assume that $(X_{(v,j)})$ and $(P_v)_{v \in \ZZ^d}$ are independent.
	Now for every $v \in \ZZ^d$ we use the Poisson random variable to determine the number of points we load into the Cube $v + [0,1]^d$ and we use the uniform random variables $X_{(v,j)}$, $j \in \NN$ to determine the position of the points to add in the Cube.
	The Poisson point process $\eta$ on $\RR^d$ is given by
	\begin{align*}
		\eta := \sum\limits_{v \in \ZZ^d} \sum\limits_{j = 1}^{P_v} \delta_{v + X_{(v,j)}},
	\end{align*}
	where $\delta_{w}$ denotes the Dirac-measure with mass concentrated in the point $w \in \RR^d$.
	The above defined measure $\eta$ is a Poisson point process on $\RR^d$ with intensity measure $\mu = t \Lambda_d$.
\end{remark}
%
%
\section{Random Simplicial Complexes}
\label{sec:rand-simpl-compl}
In this section we apply the Vietoris-Rips and the \v{C}ech complex construction
to the random set of points given by a Poisson point process $\eta$.
%
\begin{definition}
Let $\eta$ be a stationary Poisson point process on $\RR^d$. We define the Poissonized versions of the Vietoris-Rips complex and \v{C}ech complex by replacing the fixed point set in the Definitions \ref{def:vr} and \ref{def:cc} by the random point set that is given by $\eta$, i.e.
	\begin{align*}
		\vietoris(\eta, \rho) = \vietoris\rbras*{\cbras*{x \in \RR^d : \eta(\cbras{x}) > 0}, \rho},
	\end{align*}
	and
	\begin{align*}
		\cech(\eta, \rho) = \cech\rbras*{\cbras*{x \in \RR^d : \eta(\cbras{x}) > 0}, \rho}.
	\end{align*}
\end{definition}

In a first step one should be interested in the question whether this construction yields isolated bounded components or one or several connected unbounded components. This was answered by Meester and Roy \cite{meester-roy-art, meester-roy-book}.

\begin{theorem}[Percolation]\label{th:percolation}
For each $\rho>0$ there exists $t_{\text{perc}}(\rho)>0$, 
such that for every stationary Poisson point process $\eta$ on $\RR^d$ 
with intensity $t < t_{\text{perc}}(\rho)$ the random \v{C}ech complex $\cech(\eta, \rho)$ 
  contains no unbounded connected simplicial complex with probability one. If the intensity satisfies $t > t_{\text{perc}}(\rho)$ then  with probability one the random \v{C}ech complex $\cech(\eta, \rho)$ 
  contains precisely one unbounded connected simplicial complex $\beast$,
a unique giant component in $\cech(\eta, \rho)$.
\end{theorem}
%
%
%

To simplify notation, we use $\cS_{\vietoris,d}$ to denote the class of all simplicial complexes that admit a $d$-dimensional Vietoris-Rips representation. Similarly, $\cS_{\cech,d}$ denotes the class of all simplicial complexes having a $d$-dimensional \v{C}ech representation. As it will usually be clear from the context, which dimension $d$ we are considering, we will mostly omit $d$ from the notation and just write $\cS_{\vietoris}$ and $\cS_{\cech}$ in the following.

\begin{theorem}\label{thm:iso_vr}
	For every simplicial complex $K \in \cS_{\vietoris,d}$ with vertex set $V$ and parameter $\rho > 0$ and every stationary Poisson point process $\eta$ on $\RR^d$ with intensity $t > 0$ the infinite Vietoris-Rips complex $\vietoris(\eta, \rho)$ contains almost surely infinitely many isolated simplicial complexes $K_n$, $n \in \NN$ that are combinatorially equivalent to $K$.
\end{theorem}

\begin{theorem}\label{thm:perc_vr}	For every simplicial complex $K \in \cS_{\vietoris,d}$  with vertex set $V$ and parameter $\rho > 0$ and every stationary Poisson point process $\eta$ on $\RR^d$ with intensity $t > t_{\text{perc}}(\rho)$ the infinite Vietoris-Rips complex $\vietoris(\eta, \rho)$ contains a giant unbounded connected simplicial complex $\cB_\infty$ which contains almost surely infinitely many sub complexes $K_n$, $n \in \NN$ that are combinatorially equivalent to $K$ and only connected to $\cB_\infty$ by an edge.
\end{theorem}


By replacing the $\vietoris$ operator with the $\cech$ operator, both theorems can be formulated for the class of \v{C}ech complexes too.
	
\begin{theorem}\label{thm:iso_c}
	For every simplicial complex $K \in \cS_{\cech,d}$ with vertex set $V$ and parameter $\rho > 0$ and every stationary Poisson point process $\eta$ on $\RR^d$ with intensity $t > 0$ the infinite \v{C}ech complex $\cech(\eta, \delta)$ contains almost surely infinitely many isolated simplicial complexes $K_n$, $n \in \NN$ that are combinatorially equivalent to $K$.
\end{theorem}

\begin{theorem}\label{thm:perc_c}
	For every simplicial complex $K \in \cS_{\vietoris,d}$ with vertex set $V$ and parameter $\rho > 0$ and every stationary Poisson point process $\eta$ on $\RR^d$ with intensity $t > t_{\text{perc}}(\rho)$ the infinite \v{C}ech complex $\cech(\eta, \rho)$ contains a giant unbounded connected simplicial complex $\cB_\infty$ which contains almost surely infinitely many sub complexes $K_n$, $n \in \NN$ that are combinatorially equivalent to $K$ and only connected to $\cB_\infty$ by an edge.
\end{theorem}

%
%
%
%

\begin{proof}[Proof of Theorem \ref{thm:iso_vr} and \ref{thm:iso_c}]
	Let $K \in \cS_{\vietoris,d}$ with vertex set $V = \cbras*{v_0, \ldots, v_n}$ and parameter $\rho > 0$, then $K = \vietoris(V, \rho)$ by definition of $\cS_{\vietoris,d}$. 
	Further let $\eta$ be a stationary Poisson point process on $\RR^d$ with intensity $t > 0$.
	We denote by $\alpha > 0$ the smallest distance of two vertices of $K$, i.e.,
	\begin{align*}
		\alpha = \min\limits_{\substack{v,v' \in V\\v \neq v'}} \norm{v - v'}_2.
	\end{align*}
	
	It follows from Theorem \ref{thm:all_is_cr} that $K$ is continuous realizable and therefore by Definition \ref{def:distance} there exists a $\delta_0 > 0$ such that for every point set $X$ with $\dist(X,V) < \delta_0$ the complex $\vietoris(X,\rho)$ is combinatorially equivalent to $K$.
	
	We set $\delta := \min\rbras*{\frac{\alpha}{2}, \delta_0}$ to ensure, that all balls with radius $\delta$ around the vertices of $V$ are pairwise disjoint.
	
	We define the surrounding box $W$ of the set $V$ and its coordinate width $\beta$ by
	\begin{align*}
		W = \prod\limits_{i = 1}^d [a_i, b_i], \qquad \beta = \max\limits_{i = 1}^d \abs*{a_i - b_i}
	\end{align*}
	where $a_i, b_i \in \RR$ are the minimum resp. the maximum value of the $i$-th entries of the position vectors of all coordinates, i.e.,
	\begin{align*}
		a_i = \min\limits_{v \in V} {v_i}, \qquad b_i = \max\limits_{v \in V} {v_i}.
	\end{align*}
	
	For $\theta \in \gamma \ZZ^d$ with $\gamma := \beta + 2(\delta + \rho)$ we construct the two translated and extended boxes
	\begin{alignat*}{2}
		\theta + W_I & = \cbras*{\theta} + W + \delta B_1^d, &&\qquad \text{inner box},\\
		\theta + W_O & = \cbras*{\theta} + W + (\delta+\rho) B_1^d, &&\qquad \text{outer box},
	\end{alignat*}
	as Minkowski addition of $W$ with the scaled $d$-dimensional open unit ball $\delta B_1^d$ and the translation direction $\theta$.
	Further we denote by $\theta+V$ the translated vertex set $\cbras*{\theta + v : v \in V}$.	
	
	Now we can investigate the randomly generated Vietoris-Rips complex $\vietoris(\eta \cap (\theta + W_I),\rho)$ and show that it is combinatorially equivalent to $K$ with a positive probability that is not depending on the translation $\theta \in \gamma \ZZ^d$.
	
	\begin{lemma}\label{lem:iso_a_theta}
		Let $A_{\theta}$ be the event, that the Vietoris-Rips complex $\vietoris(\eta \cap (\theta + W_I),\rho)$ is combinatorially equivalent to $K$, i.e.,
		\begin{align*}
				A_{\theta} := \cbras*{ \vietoris(\eta \cap (\theta + W_I)) \text{ is combinatoric equivalent to } K}.
		\end{align*}
		Then there exists a constant $c_A \in (0, 1]$ such that $\Prob*{A_{\theta}} \geq c_A$ for all $\theta \in \gamma \ZZ^d$.
		\begin{proof}
			From Definition \ref{def:distance} it follows directly that if $\dist(\eta \cap (\theta + W_I), \theta + V) < \delta$ holds, than the simplicial complexes $\vietoris(\eta \cap (\theta + W_I),\rho)$ and $K$ are combinatorially equivalent. Therefore 
			\begin{align*}
				A_{\theta} \supseteq \cbras*{ \dist(\eta \cap (\theta + W_I), \theta + V) < \delta },
			\end{align*}
			where the right hand side denotes the event, that in every open ball with radius $\delta$ around every translated vertex $v \in \theta+V$ lays exactly one point of $\eta$ and that there are no points in $\eta \cap (\theta + W_I)$ that are laying outside of these balls. Further by the choice of $\delta$ all balls are pairwise disjoint,
			\begin{align*}
				A_{\theta} \supseteq \cbras*{ \eta\rbras*{B_\delta^d(v)} = 1, \forall v \in V \text{ and } \eta \rbras*{W_I \setminus \bigcup\limits_{v \in V} B_\delta^d(v) } = 0},
			\end{align*}
			where we used the translation in-variance of $\eta$ to set $\theta = 0$. By the independence property of our Poisson point process $\eta$ it follows that
			\begin{align*}
				\Prob*{A_{\theta}} & \geq \rbras*{\prod_{v \in V } \Prob*{\eta\rbras*{B_\delta^d(v)} = 1}} \cdot \Prob*{\eta \rbras*{W_I \setminus \bigcup\limits_{v \in V } B_\delta^d(v) } = 0}\\
								& = \prod_{v \in V} \rbras*{t\kappa_d \delta^d \exp(-t\kappa_d \delta^d)} \cdot \exp\rbras*{- \rbras*{t \Lambda_d(W_I) - \abs{V } t  \kappa_d \delta^d}}\\
								& = \rbras*{t \kappa_d \delta^d}^{n+1}  \exp \rbras*{-(n+1) t \kappa_d \delta^d - t\Lambda_d(W_I) + (n+1) t \kappa_d \delta^d}\\
								& = \rbras*{t \kappa_d \delta^d}^{n+1} \exp\rbras*{- t \Lambda_d (W_I)} =: c_A > 0,
			\end{align*}
			where $\kappa_d$ denotes the volume of the $d$-dimensional unit ball.
		\end{proof}
	\end{lemma}
	
	In the next step we will define the event $B_{\theta}$ to ensure that the simplicial complex given by $\vietoris(\eta \cap (\theta + W_I),\rho)$ is isolated.
	Again we will show, that this event has a positive probability not depending on the translation $\theta \in \gamma \ZZ^d$.
	
	\begin{lemma}\label{lem:iso_b_theta}
		Let $B_{\theta}$ be the event, that the Poisson point process $\eta$ has no points in the $\rho$-neighborhood around $\theta + W_I$, i.e.,
		\begin{align*}
			B_{\theta} := \cbras{ \eta\rbras*{(\theta + W_O) \setminus (\theta + W_I)} = 0}.
		\end{align*}
		Then there exists a constant $c_B \in (0,1]$ such that $\Prob*{B_{\theta}} = c_B$ for all $\theta \in \gamma \ZZ^d$.
		\begin{proof}
			By definition we have $W_I \subset W_O$ and therefore it follows from the translation in-variance that
			\begin{align*}
				\Prob*{B_{\theta}} = \exp \rbras*{-t \rbras*{ \Lambda_d (W_O) - \Lambda_d (W_I) }} =: c_B > 0,
			\end{align*}
			for all $\theta \in \gamma\ZZ^d$.
		\end{proof}
	\end{lemma}
	
	To prove our main result, it is sufficient to show that there existing infinitely many translations $\theta \in \gamma \ZZ^d$ such that $E_{\theta} := A_{\theta} \cap B_{\theta}$ occurs. The event $A_{\theta}$ states that $\vietoris(\eta \cap (\theta + W_I),\rho)$ is combinatorially equivalent to $K$ and $B_{\theta}$ ensures that this simplicial complex can not be connected to any other point that is not already in the vertex set $\eta \cap (\theta + W_I)$. Note that for all $y \in \eta \setminus (\theta + W_I)$ and for all $x \in \eta \cap (\theta + W_I)$ we have $\norm{y - x}_2 > \rho$.
	
	Further it follows from the independence property of $\eta$ that the events $A_{\theta}$ and $B_{\theta}$ are independent and therefore by Lemma \ref{lem:iso_a_theta} and \ref{lem:iso_b_theta} we have
	\begin{align*}
		\Prob*{E_{\theta}} = \Prob*{A_{\theta}}\Prob*{B_{\theta}} \geq c_A c_B =: c \in (0, 1]
	\end{align*}
	for all $\theta \in \gamma \ZZ^d$.
	Now we are ready to use the well known Borel-Cantelli lemma:
	\begin{lemma}[Borel-Cantelli]
		Let $E_1, E_2, \ldots$ be a sequence of pairwise independent events on some probability space $(\Omega, \sF, \dsP)$ with $\sum\limits_{i = 1}^\infty \Prob*{E_i} = \infty$.
		
		Then $\Prob*{\limsup\limits_{i \rightarrow \infty} E_i} = 1$.
	\end{lemma}
	
	For $\theta_1, \theta_2 \in \gamma \ZZ^d$ with $\theta_1 \neq \theta_2$ it follows directly by the definition of $\gamma$ that $(\theta_1 + W_O) \cap (\theta_2 + W_O) = \emptyset$ and therefore the events $E_{\theta_1}$ and $E_{\theta_2}$ are independent. Further there exists a constant $c := c_A c_B \in (0,1]$ such that $\Prob*{E_{\theta}} \geq c$ for all $\theta \in \gamma \ZZ^d$. 
	
	Note that $\ZZ^d$ is countable, thus we can apply a bijection $\varphi : \NN \rightarrow \gamma \ZZ^d$ to derive that
	\begin{align*}
		\sum\limits_{\theta \in \gamma \ZZ^d} \Prob*{E_\theta} = \sum\limits_{i = 1}^\infty \Prob*{E_{\varphi(i)}} \geq \sum\limits_{i = 1}^\infty c = \infty.
	\end{align*}
	Using the Borel-Cantelli Lemma on the series $E_{\varphi(1)}, E_{\varphi(2)}, \ldots$ yields
	\begin{align*}
		\Prob*{\limsup\limits_{i \rightarrow \infty} E_{\varphi(i)}} = 1,
	\end{align*}
	and thus
	\begin{align*}
		& \Prob*{E_{\varphi(i)} \text{ for infinitly many } i \in \NN }\\ 
		= &\Prob*{E_{\theta} \text{ for infinitly many translations } \theta \in \gamma \ZZ^d} = 1,
	\end{align*}
	which completes the proof of Theorem \ref{thm:iso_vr}. The proof of \ref{thm:iso_c} is similar.
\end{proof}

The proofs of Theorem \ref{thm:iso_vr} and Theorem \ref{thm:iso_c} use the Borel-Cantelli Lemma in its original form and thus rely essentially on the strong independence property of the Poisson point process. In the following proof of Theorem \ref{thm:perc_vr} and Theorem \ref{thm:perc_c} this is no longer possible. If two complexes are both joined via the giant component of the simplicial complex, they are not independent and a careful analysis of their dependency structure is necessary.

\begin{proof}[Proof of Theorem \ref{thm:perc_vr} and \ref{thm:perc_c}]
	Let $\eta$ be a stationary Poisson point process on $\RR^d$ with intensity $t > t_{\text{perc}}(\rho)$.
	Note that in this case the lower bound for the intensity is depending on the distance parameter of the simplicial complex, to ensure that the $1$-skeleton of $\vietoris(\eta, \rho)$ percolates and thus there exists a giant unbounded connected component, denoted as $\cB_\infty$, in $\vietoris(\eta, \rho)$, see Theorem \ref{th:percolation}.
	
	For $K \in \cS_{\vietoris,d}$ with vertex set $V = \cbras*{v_0, \ldots, v_n}$ and parameter $\rho > 0$, we have $K = \vietoris(V, \rho)$ by definition of $\cS_{\vietoris,d}$.
	Denote by $\conv(V)$ the convex hull of the set $V$ and choose one vertex on the boundary of the convex hull as the connection vertex ${v_A \in \partial \conv(V) \cap V}$ to $\cB_\infty$.
	Let $\bbH$ be a $(d-1)$-dimensional supporting hyperplane such that $\bbH \cap \conv(V) \subseteq \partial \conv(V)$ and denote by $\bbH^\pm$ the corresponding open half spaces such that ${\bbH^+ \cap \conv(V) = \emptyset}$. Let $\vec{u} \in \RR^d$ be the normal vector of $\bbH$ pointing into the half space $\bbH^+$ with $\norm{\vec{u}}_2 = 1$. To shorten our notation, we set $V^* := V \setminus \cbras*{v_A}$.
	
	For $\varepsilon \in (0, \frac{\rho}{4})$ we define the linkage vertex $v_L \in \bbH^+$ by $v_L := v_A + (\rho - 2\varepsilon) \vec{u}$ and for all $x \in B_\varepsilon^d(v_A)$ and $y \in B_\varepsilon^d(v_L)$ it follows directly by the Triangle inequality, that $\norm{x - y}_2 < \rho$.
	Note that, by continuity, there exists an $\varepsilon_0 > 0$ small enough, such that for all $y \in B_{\varepsilon_0}^d(v_L)$ it holds that $B_\rho^d(y) \cap B_{\varepsilon_0}^d(v) = \emptyset$ for all $v \in V^*$.
		
	At this point, we conclude three important consequences of our construction by choosing the parameter $\delta$ in a proper way:
	
	\begin{remark}\label{rem:choose_delta}
	Let $\delta := \min(\frac{\alpha}{2}, \delta_0, \varepsilon_0)$, where $\alpha$ and $\delta_0$ are as in the proof of Theorem \ref{thm:iso_vr} before, then
		\begin{enumerate}[(i)]
			\item all balls with radius $\delta$ around all vertices of $V$ are pairwise disjoint,
			\item for all sets $X$ with $\dist(X,V) < \delta$ it follows that $\vietoris(X,\rho)$ is combinatorially equivalent to $K$, and
			\item every point $x \in B_\delta^d(v_L)$ is forced to connect to every point $y \in B_\delta^d(v_A)$, but not allowed to connect to any other point $z \in B_\delta^d(v)$ for all  $v \in V^*$.
		\end{enumerate}
	\end{remark}
	
	We continue with the geometric construction to prove our main result by defining the extended convex hulls given by
	\begin{alignat*}{2}
		U_I & = \conv(V) + \delta B_1^d, &&\qquad \text{inner hull},\\
		U_O & = \conv(V) + (\delta+\rho) B_1^d, &&\qquad \text{outer hull},
	\end{alignat*}
	and the two surrounding boxes $W_I \subset W_O$ of the outer hull $U_O$ that have their vertices on the lattice $s \ZZ^d$ with $s := \frac{\rho}{4\sqrt{d}}$ given by
	\begin{align*}
		\quad W_I = \prod\limits_{i = 1}^d [a_i, b_i], \qquad W_O = \prod\limits_{i=1}^d [-s k +  a_i, b_i + s k],
	\end{align*}
	where $k \in \NN$, $k > 4 \sqrt{d}$ and $a_i, b_i \in s \ZZ$ are given by
	\begin{align*}
		a_i = \max\limits_{y \in s\ZZ}\cbras*{y : y \leq \min\limits_{x \in U_O} \cbras{x_i}}, \quad b_i = \min\limits_{y \in s\ZZ}\cbras*{y : y \geq \max\limits_{x \in U_O} \cbras{x_i}}.
	\end{align*}
	Further we define the "hitbox" ~$W_H$ as extension of the box $W_O$ by
	\begin{align*}
		W_H = W_O + \frac{\rho}{2\sqrt{d}} B_1^d.
	\end{align*}
	
	We denote by $\gamma_i$ the coordinate width of $W_H$ in direction of the $i$-th standard basis vector $\vec{e}_i$. Clearly, by definition of $W_H$ it holds that $\gamma_i > \rho + \frac{\rho}{\sqrt{d}}$ for all $i \in \cbras*{1, \ldots, d}$ and additionally $\gamma_i$ is a integer multiple of the lattice width $s = \frac{\rho}{4 \sqrt{d}}$.
	
	No we can conclude two important properties of our lattice $s \ZZ^d$ in the following lemmas:
	
	\begin{lemma}\label{lem:lattice_independence}
		For all $x,y \in s\ZZ^d$ with $x \neq y$ the events $\cbras*{ \eta\rbras*{B^d_\frac{s}{2}(x)} = 1}$ and $\cbras*{ \eta\rbras*{B^d_\frac{s}{2}(y)} = 1}$ are independent.
		\begin{proof}
			It follows directly from $x \neq y$, that the distance between the lattice points is bounded from below by the lattice width, thus we have $\norm{x - y}_2 \geq s$ and $B^d_\frac{s}{2}(x) \cap B^d_\frac{s}{2}(y) = \emptyset$.
			The claim then follows directly from the independence property of the Poisson point process $\eta$.
		\end{proof}
	\end{lemma}
	
	\begin{lemma}\label{lem:lattice_covering}
		Let $L^*$ be the set of all lattice points $z \in s \ZZ^d$ that are laying in the interior of the shell $W_O \setminus W_I$.
		Let $X \subset \RR^d$ be a set of points, such that $\dist(X, L^*) < \frac{s}{2} = \frac{\rho}{8\sqrt{d}}$, then the balls with radius $\rho$ around the points of $X$ are covering the inner shell $W_O \setminus W_I$ and the outer shell $W_H \setminus W_O$, i.e.
		\begin{align*}
			W_H \setminus W_I = (W_O \setminus W_I) \cup (W_H \setminus W_O) \subset \bigcup\limits_{x \in X} B^d_\rho(x)
		\end{align*}
		\begin{proof}
			First, let $y \in W_O \setminus W_I$ be a point in the inner shell, then by definition of $L^*$ there exists a lattice point $z \in L^*$ such that $y = z + v$ for some $v$ with $\norm{v}_2 \leq \sqrt{d} s = \frac{\rho}{4}$. Further, for every lattice point $z \in L^*$ there exists a unique point $x \in X$ such that $\norm{z - x}_2 < \frac{s}{2} = \frac{\rho}{8\sqrt{d}}$, and thus it follows by the Triangle inequality that $\norm{y - x}_2 \leq \frac{\rho}{4} + \frac{\rho}{8\sqrt{d}} < \rho$. Second, let $y' \in W_H \setminus W_O$ be a point in the outer shell, then by definition of $W_H$ there exists a point $y \in W_O \setminus W_I$ such that $\norm{y - y'}_2 \leq \frac{\rho}{2\sqrt{d}}$. Thus $\norm{y' - x}_2 \leq \norm{y' - y}_2 + \norm{y - x}_2 < \frac{\rho}{2\sqrt{d}} + \frac{\rho}{4} + \frac{\rho}{8\sqrt{d}} < p$. Finally the desired result follows directly from these two steps.
		\end{proof}
	\end{lemma}
	
	For $\theta \in \gamma \ZZ^d := \gamma_1 \ZZ \times \ldots \times \gamma_d \ZZ$ we define the translated sets $(\theta + U_I)$, $(\theta + U_O)$, $(\theta + W_I)$, $(\theta + W^{(k)}_O)$, $V(\theta)$ and $(\theta + V^*)$ as Minkowski addition of $\cbras*{\theta}$ to the corresponding set. Also we define the translated connection point $\theta + v_A$ and linkage point $\theta + v_L$, but to shorten the Notation we will omit to specify $\theta$ in the following, when all sets and points are considered to be translated with the same vector $\theta \in \gamma \ZZ^d$.
	Further, by using the translation in-variance of our Poisson point process $\eta$ the following results are not depending on the choice of $\theta \in \gamma \ZZ^d$, thus w.l.o.g. we can set $\theta = 0$ to simplify the following proofs.
	
	To make use of Lemma \ref{lem:lattice_covering} we will now define the event, that the set of points in the Poisson point process, that are laying in the shell $W_O \setminus W_I$ has a distance smaller than $\frac{s}{2}$ to the set $L^*$ and show that this event has a positive probability that is not depending on the translation $\theta \in \gamma \ZZ^d$.
	
	\begin{lemma}
		Let $B_\theta$ be the be the event, that the set $\eta \cap (\theta + (W_O \setminus W_I))$ has distance smaller than $\frac{s}{2}$ to the set $\theta + L^*$, i.e.,
		\begin{align*}
			B_\theta := \cbras*{\dist\rbras*{\eta \cap (\theta + (W_O \setminus W_I)), \theta + L^*} < \frac{s}{2}}.
		\end{align*}
		Then there exists a constant $c_B \in (0,1]$ such that $\Prob*{B_\theta} \geq c_B$ for all $\theta \in \gamma \ZZ^d$.
		\begin{proof}
			Note that $L^*$ is a finite set of points and denote by $l := \abs{L^*}$ the cardinality of $L^*$. 
			Similar to the proof of Lemma \ref{lem:iso_a_theta} we can rewrite the event $B_\theta$ and by using the independence gained from Lemma \ref{lem:lattice_independence} we have
			\begin{align*}
				\Prob*{B_\theta} = \Prob*{B_0} & = \rbras*{ \prod\limits_{x \in L^*} \Prob*{\eta(B^d_\frac{s}{2}(x)) = 1}} \cdot \Prob*{\eta\rbras*{(W_O \setminus W_I) \setminus \bigcup\limits_{x \in L^*} B^d_\frac{s}{2}(x)}}\\
								& = \rbras*{t \kappa_d (\tfrac{s}{2})^d}^l \exp\rbras*{-l t \kappa_d (\tfrac{s}{2})^d - t \Lambda_d(W_O \setminus W_I) + l t \kappa_d (\tfrac{s}{2})^d}\\
								& = \rbras*{t \kappa_d (\tfrac{s}{2})^d}^l \exp\rbras*{-t \Lambda_d(W_O \setminus W_I)} =: c_B > 0.
			\end{align*}
		\end{proof}
	\end{lemma}
	
	As in the proof of Theorem \ref{thm:iso_vr} we can  show that the randomly generated simplicial complex $\vietoris(\eta \cap (\theta + U_I), \rho)$ is combinatorially equivalent to $K$ with a positive probability that is not depending on the translation $\theta \in \gamma \ZZ^d$.
	
	\begin{lemma}\label{thm:main:perc:proof:lem:a_theta}
		Let $A_{\theta}$ be the event, that the Vietoris-Rips complex $\vietoris(\eta \cap (\theta + U_I),\rho)$ is combinatorially equivalent to $K$, i.e.,
		\begin{align*}
				A_{\theta} := \cbras*{ \vietoris(\eta \cap (\theta + U_I), \rho) \text{ is combinatorially equivalent to } K}.
		\end{align*}
		Then there exists a constant $c_A \in (0, 1]$ such that $\Prob*{A_{\theta}} \geq c_A$ for all $\theta \in \gamma \ZZ^d$.
		\begin{proof}
			Note that $\delta < \frac{\rho}{4}$ ensures that $B_\delta^d(v_L) \cap U_I = \emptyset$ and therefore, by using the same arguments as in the proof of Lemma \ref{lem:iso_a_theta}, we obtain the desired result for 
			\begin{align*}
				c_A := \rbras*{t \kappa_d \delta^d}^{n+1} \exp\rbras*{- t \Lambda_d (U_I)}.
			\end{align*}
		\end{proof}
	\end{lemma}
	
	In the next step we will ensure, that the randomly generated simplicial complex $\vietoris(\eta \cap U_I, \rho)$ is forced to connect to $\eta \cap (W_O \setminus W_I)$ by the attachment vertex $v_A$ and the linkage vertex $v_L$. 
	Therefore we define the line $v_L(\zeta) = v_L + \zeta \vec{u}$, $\zeta \in [0, \infty)$ as continuation of the direct path from $v_A$ to $v_L$. 
	Note that we can choose a finite number of points $v_L(0) = v_L(\zeta_0), v_L(\zeta_1), \ldots, v_L(\zeta_N)$ on the line such that 
	\begin{enumerate}
		\item $v_L(\zeta_i) \in W_I$ for all $i = 0, \ldots, N$,
		\item it holds that $\norm{v_L(\zeta_i) - v_L(\zeta_{i-1})}_2 < \rho - 2 \delta$ for all $i = 1, \ldots, N$,
		\item $0 < \norm{v_L(\zeta_N) - z}_2 < \frac{\rho}{2\sqrt{d}}$ for some $z \in \partial W_I$.
	\end{enumerate}
	
	We will now use the translated versions of these points to construct the linkage from the random simplicial complex to the inner shell by defining the event, that in every ball with a positive but small radius $\delta'$ around these linkage points is exactly one point of the Poisson point process and that the remaining space in the inner box, that is not already covered by $U_I$ or by these balls is empty.
	Denote by $VL = \cbras*{v_L(0), v_L(\zeta_1), \ldots, v_L(\zeta_N)}$ the set of linkage points and choose
	\begin{align*}
		\delta' := \min\cbras*{ \delta, ~\min\limits_{z \in \partial W_I} \norm{v_L(\zeta_N) - z}_2, ~\frac{1}{2} \min\limits_{\substack{x,y \in VL \\ x \neq y }}\norm{x-y}_2 },
	\end{align*}
	to ensure, that all balls with radius $\delta'$ around the points in $VL$ are pairwise disjoint and laying completely in the inner shell $W_I \setminus U_I$. 
	Then we can define the event $A'_\theta$ to ensure that if $A'_\theta$ occurs, the connection from $\theta + v_A$ to the boundary of $\theta + W_I$ is established and we show that this event has a positive probability that is not depending on the translation $\theta \in \gamma \ZZ^d$:
	
	\begin{lemma}
		Let $A'_\theta$ denote the event that all points of $\eta$ that are laying in the space $\theta + (W_I \setminus U_I)$ are $\delta'$ close to the set $\theta + VL$, i.e.
		\begin{align*}
			A'_\theta := \cbras*{ \dist(\eta \cap (\theta + (W_I \setminus U_I)), \theta + VL) < \delta'}.
		\end{align*}
		Then there exists a constant $c_A' \in (0,1]$ such that $\Prob*{A'_\theta} \geq c_A'$ for all $\theta \in \gamma \ZZ^d$.
		\begin{proof}
			The cardinality of $VL$ is given by $N+1$ and the choice of $\delta'$ ensures, that we can rewrite this event in a similar manner like in the proof of the lemma before. It follows that
			\begin{align*}
				\Prob*{A'_\theta} = \Prob*{A'_0} & = \rbras*{ \prod\limits_{x \in VL} \Prob*{\eta\rbras*{B^d_{\delta'}(x)} = 1 } } \cdot \Prob*{\eta\rbras*{ (W_I \setminus U_I) \setminus \bigcup\limits_{x \in VL} B^d_{\delta'}(x)}}\\
								& = \rbras*{t \kappa_d \delta'^d}^{N+1} \exp\rbras*{- (N+1) t \kappa_d \delta'^d - t \Lambda_d(W_I \setminus U_I) + (N+1) t \kappa_d \delta'^d}\\
								& = \rbras*{t \kappa_d \delta'^d}^{N+1} \exp\rbras*{- t \Lambda_d(W_I \setminus U_I)} =: c'_A.
			\end{align*}
		\end{proof}
	\end{lemma}
	
	We conclude the results of the previous construction in the following lemma:
	
	\begin{lemma}\label{lem:c_T}
		For $\theta \in \gamma \ZZ^d$ we define the event $T_\theta$ that the simplicial complex $\vietoris(\eta \cap (\theta + U_I), \rho)$ is combinatorially equivalent to $K$ and connected to a set of points in the outer shell $\theta + (W_O \setminus W_I)$ that are the center points of balls with radius $\rho$ that are a covering of the hitbox $\theta + W_H$. Then there exists a constant $c_T \in (0, 1]$ such that $\Prob*{T_\theta} \geq c_T$ for all $\theta \in \gamma \ZZ^d$.
		\begin{proof}
			It follows directly from our construction that the occurrence of $A_\theta$, $A_\theta'$ and $B_\theta$ at the same translation $\theta$ implies that also $T_\theta$ occurs at this translation $\theta$. Therefore it follows that
			\begin{align*}
				T_\theta \supseteq A_\theta \cap A_\theta' \cap B_\theta.
			\end{align*}
			Further the events $A_\theta$, $A_\theta'$ and $B_\theta$ are defined on the pairwise disjoint sets $\theta + U_I$, $\theta + (W_I \setminus U_I)$ and $\theta + (W_O \setminus W_I)$, thus it follows directly by the independence property of the Poisson point process $\eta$ that
			\begin{align*}
				\Prob*{T_\theta} \geq \Prob*{A_\theta \cap A_\theta' \cap B_\theta} = \Prob*{A_\theta} \cdot \Prob*{A_\theta'} \cdot \Prob*{B_\theta} \geq c_A c_A'c_B =: c_T.
			\end{align*}
		\end{proof}
	\end{lemma}	
	
	To continue with the proof, we will investigate the event, that the giant unbounded component $\cB_\infty$ "meets" the box $W_H$ in a way, such that it is forced to connect to the points inside the box $\theta + W_H$, given $T_\theta$ occurs. We can show, that the event $H_\theta$ has a positive probability that is not depending on the translation $\theta \in \gamma \ZZ^d$:
	
	\begin{lemma}\label{lem:c_H}
		Denote by $H_\theta$ the event, that there exists a vertex of $\cB_\infty$ laying in $\theta + (W_H \setminus W_O)$. Then there exists a constant $c_H \in (0,1]$ such that $\Prob*{H_\theta} \geq c_H$ for all $\theta \in \gamma \ZZ^d$.
		\begin{proof}
			Assume, that $c_H=0$, then it follows from the translation in-variance of $\eta$ that $\Prob*{H_\theta} =  0$ for all translations $\theta \in \RR^d$. Further we can find a countable index set $I \subset \RR^d$ of translations such that $\bigcup_{\theta \in I} \rbras*{\cbras*{\theta} + (W_H \setminus W_O)} = \RR^d$, thus it follows
			\begin{align*}
				\Prob*{\exists x \in \cB_\infty} & = \Prob*{\exists x \in \cB_\infty \cap \bigcup\limits_{\theta \in I} \rbras*{\cbras*{\theta} + (W_H \setminus W_I)}}\\
				& \leq \Prob*{\exists x \in \bigcup\limits_{\theta \in I} \rbras*{ \cB_\infty \cap \rbras*{\cbras*{\theta} + (W_H \setminus W_I)}}}\\
				& \leq \sum\limits_{\theta \in I} \Prob*{\exists x \in \cB_\infty \cap \rbras*{\cbras*{\theta} + (W_H \setminus W_I)}} = \sum\limits_{\theta \in I} 0 = 0,
			\end{align*}
			which implies, that there does not exists a giant unbounded component in $\RR^d$ which is false, because we are in the case of percolation.
		\end{proof}
	\end{lemma}
	
	The construction of $\theta + W_H$ implies, that if $H_\theta$ and $T_\theta$ occur for the same translation $\theta \in \gamma \ZZ^d$, then the simplicial complex $\vietoris(\eta \cap (\theta + U_I), \rho)$ is combinatorially equivalent to $K$ and that it is connected to the giant unbounded component $\cB_\infty$.
	Note that the event $H_\theta$ depends only on the points of $\eta$ laying outside of $\theta + W_O$, and the event $T_\theta$ depends only on the points of $\eta$ inside of $\theta + W_O$, thus the events $H_\theta$ and $T_\theta$ are independent.
	In the remaining part of this proof, we will show, that the event $E_\theta = H_\theta \cap T_\theta$ occurs for infinitely many translations $\theta \in \gamma \ZZ^d$.
	Therefore we will use the following modified version of Borel-Cantelli that is stated in \cite{ReSchneider}:
	\begin{lemma}[Modified Borel-Cantelli]\label{lem:mod_borel_cantelli}
		Let $E_1, E_2, \ldots $ be a sequence of events on some probability space $(\Omega, \sF, \dsP)$ with $\sum\limits_{i = 1}^\infty \Prob*{E_i} = \infty$ and
		\begin{align}\label{eq:mod_borel_cantelli}
			\liminf\limits_{n \rightarrow \infty} \frac{\sum_{i,j = 1, i \neq j}^n \sbras*{\Prob*{E_i \cap E_j} - \Prob*{E_i} \Prob*{E_j}}}{\rbras*{\sum_{i = 1}^n \Prob*{E_j}}^2} = 0.
		\end{align}
		Then $\Prob*{\limsup\limits_{i \rightarrow \infty} E_i} = 1$.
	\end{lemma}

	Note that $E_{\theta_1}$ and $E_{\theta_2}$ are not independed for $\theta_1, \theta_2 \in \gamma \ZZ^d$, but there exists a series of translations $(\theta_i)_{i \in \NN_0}$ such that \eqref{eq:mod_borel_cantelli} holds for the Events $(E_{\theta_i})_{i \in \NN_0}$.
	
	\begin{lemma}\label{lem:prob_dependence_bound}
		Let $\theta_0 := 0$ and $\theta_i := \theta_{i-1} + \gamma_1 e_1 k_i$, then there exist a series of integers $k_i \in \NN$ such that
		\begin{align*}
			\abs{\Prob*{E_i \cap E_j} - \Prob*{E_i}\Prob*{E_j}} \leq \frac{1}{i^2},
		\end{align*}
		for all $i,j \in \NN_0$ with $j > i$.
		\begin{proof}
			For $i \in \NN$ and sufficiently large $R > 0$ we define the translation $\theta_R := \theta_{i-1} + R e_1$ and the two half-spaces
			\begin{align*}
				\bbH^- := \cbras*{\sskp*{x,e_1} \leq \frac{\norm*{\theta_{i-1} + \theta_R}}{2}},
			\end{align*}
			and 
			\begin{align*}
				\bbH^+ := \cbras*{\sskp*{x,e_1} > \frac{\norm*{\theta_{i-1} + \theta_R}}{2}},
			\end{align*}
			where $e_1$ denotes the first element of the standard basis of $\RR^d$. Note that $\bbH^-$ and $\bbH^+$ are delimited by the hyperplane that is orthogonal to the connecting line from $\theta_{i-1}$ to $\theta_R$ and contains the center of this line.
			
			We denote by $\eta^\pm := \eta \cap \bbH^\pm$ the restrictions of $\eta$ to the half-spaces. 
			Observe that $\eta^+$ and $\eta^-$ are two independent Poisson point processes and since there is percolation, we have an unbounded connected component $C^\pm$ in each halfspace.
			Further, these unbounded components $C^\pm$ connect to the single giant component $\cB_\infty$ in the space $\RR^d = \bbH^+ \cup \bbH^-$.
			
			Denote by $D^-$ the event, that $\theta_i + W_H$ connects to $C^-$ withing $\bbH^-$ and by $D^+$ the event, that $\theta_R + W_H$ connects to $C^+$ within $\bbH^+$.
			Then the independence property of the Poisson point process first shows that $D^+$ and $D^-$ are independent, and further that
			\begin{align}\label{eq:d_independence}
				\Prob*{E_{\theta_i} \cap E_{\theta_R} \cap (D^- \cap D^+)} = \Prob*{E_{\theta_i} \cap D^-}\Prob*{E_{\theta_R}  \cap D^+}.
			\end{align}
			Furthermore the event $E_{\theta_i}$ occurs if $\theta_i + W_H$ connects to $C^- \cup C^+$. 
			If in addition $D^-$ holds, then it connects even within $C^-$. 
			In other words, $E_{\theta_i} \cap (D^-)^C$ only occurs if $\theta_i + W_H$ connects to $C^- \cup C^+$ via $C^+$. 
			But this implies, that there is a component $\tilde{C}$ connecting $\theta_i + W_H$ to $\bbH^+$ and avoiding $C^-$. 
			For $R \rightarrow \infty$ this means that $\tilde{C}$ is a second unbounded connected component in the half space $\bbH^-$ which avoids $C^-$, which has probability zero.
			Hence
			\begin{align}\label{eq:d_m_limit}
				\lim\limits_{R \rightarrow \infty}\Prob*{E_{\theta_i} \cap (D^-)^C} = 0
			\end{align}
			and by the translation and rotation in-variance it follows, that
			\begin{align}\label{eq:d_p_limit}
				\lim\limits_{R \rightarrow \infty} \Prob*{E_{\theta_R} \cap (D^+)^C} = 0.
			\end{align}
			Combining these results yields
			\begin{align*}
				&\abs*{\Prob*{E_{\theta_i} \cap E_{\theta_R}} - \Prob*{E_{\theta_i}} \Prob*{E_{\theta_R}}}\\
				= & \left\vert \Prob*{E_{\theta_i} \cap E_{\theta_R} \cap (D^- \cap D^+)} + \Prob*{E_{\theta_i} \cap E_{\theta_R} \cap (D^- \cap D^+)^C} \right.\\
				& \left. - \rbras*{\Prob*{E_{\theta_i} \cap D^-} + \Prob*{E_{\theta_i} \cap (D^-)^C}} \rbras*{\Prob*{E_{\theta_R} \cap D^+} + \Prob*{E_{\theta_R} \cap (D^+)^C}}\right\vert\\
				\leq & \abs*{\Prob*{E_{\theta_i} \cap E_{\theta_R} \cap (D^- \cap D^+)} - \Prob*{E_{\theta_i} \cap D^-}\Prob*{E_{\theta_R}  \cap D^+}}\\
					& + \Prob*{E_{\theta_i} \cap E_{\theta_R} \cap (D^- \cap D^+)^C} + \Prob*{E_{\theta_i} \cap D^-}\Prob*{E_{\theta_R} \cap (D^+)^C}\\
					& + \Prob*{E_{\theta_i} \cap (D^-)^C}\Prob*{E_{\theta_R} \cap D^+} + \Prob*{E_{\theta_i} \cap (D^-)^C}\Prob*{E_{\theta_R} \cap (D^+)^C}.
			\end{align*}
			Using \eqref{eq:d_independence} eliminates the first term. Further we can bound one factor in each of the last three terms by $1$ and rewrite the complement in the second term which leads us to
			\begin{align*}
				&\abs*{\Prob*{E_{\theta_i} \cap E_{\theta_R}} - \Prob*{E_{\theta_i}} \Prob*{E_{\theta_R}}}\\
				\leq & \Prob*{E_{\theta_i} \cap E_{\theta_R} \cap ((D^-)^C \cup (D^+)^C)}\\
					& + \Prob*{E_{\theta_R} \cap (D^+)^C} + 2\Prob*{E_{\theta_i} \cap (D^-)^C}\\
				\leq & \Prob*{E_{\theta_i} \cap E_{\theta_R} \cap (D^-)^C} + \Prob*{E_{\theta_i} \cap E_{\theta_R} \cap (D^+)^C}\\
					& + \Prob*{E_{\theta_R} \cap (D^+)^C} + 2\Prob*{E_{\theta_i} \cap (D^-)^C}\\
				\leq & \Prob*{E_{\theta_i} \cap (D^-)^C} + \Prob*{ E_{\theta_R} \cap (D^+)^C}\\
					& + \Prob*{E_{\theta_R} \cap (D^+)^C} + 2\Prob*{E_{\theta_i} \cap (D^-)^C}\\
				\leq & 2\Prob*{E_{\theta_R} \cap (D^+)^C} + 3\Prob*{E_{\theta_i} \cap (D^-)^C}.
			\end{align*}
			Thus \eqref{eq:d_m_limit} and \eqref{eq:d_p_limit} imply
			\begin{align*}
				\lim\limits_{R \rightarrow \infty} \abs*{\Prob*{E_{\theta_i} \cap E_{\theta_R}} - \Prob*{E_{\theta_i}} \Prob*{E_{\theta_R}}} = 0.
			\end{align*}
			Therefore, for every $i \in \NN_0$ we can find a constant $R_i \in \RR_{>0}$ such that $\abs*{\Prob*{E_{\theta_i} \cap E_{\theta_R}} - \Prob*{E_{\theta_i}} \Prob*{E_{\theta_R}}} < \tfrac{1}{i^2}$ for every $R > R_i$. 
			Define $k_i := \ceil*{\tfrac{R_i}{\gamma}}$, then the claim holds for all $i, j \in \NN_0$ with $j > i$.
		\end{proof}
	\end{lemma}
	
	Finally we can apply Lemma \ref{lem:mod_borel_cantelli}, the modified version of Borel-Cantelli, to the sequence of events $(E_{\theta_i})_{i \in \NN_0}$ given by the translation sequence $(\theta_i)_{i \in \NN_0}$ constructed in Lemma \ref{lem:prob_dependence_bound}.
	To shorten the notation, we set $E_i := E_{\theta_i}$.
	
	Let $c_E := c_T c_H$, then it follows directly from the independence of $T_\theta$ and $H_\theta$ together with Lemma \ref{lem:c_T} and Lemma \ref{lem:c_H} that
	\begin{align*}
		\sum\limits_{i = 1}^n \Prob*{E_i} \geq  n c_E,
	\end{align*}
	and thus
	\begin{align*}
		\sum\limits_{i = 1}^\infty \Prob*{E_i} = \infty.
	\end{align*}
	Further, using Lemma \ref{lem:prob_dependence_bound} and the symmetry in $i$ and $j$ we have
	\begin{align*}
		\sum\limits_{i,j = 1, i \neq j}^n \abs*{\Prob*{E_i \cap E_j} - \Prob*{E_i} \Prob*{E_j}} & = 2\sum\limits_{i = 1} ^n \sum\limits_{j = i + 1}^n \abs*{\Prob*{E_i \cap E_j} - \Prob*{E_i} \Prob*{E_j}}\\
			 \leq 2\sum\limits_{i = 1}^n \sum\limits_{j = i+1}^n \frac{1}{i^2}\\
			 \leq 2n \sum\limits_{i = 1}^n \frac{1}{i^2} \leq \frac{\pi^2}{3} n.
	\end{align*}
	Thus for $c_E := c_T c_H$ it follows that
	\begin{align*}
		\frac{\sum_{i,j = 1, i \neq j}^n \abs*{\Prob*{E_i \cap E_j} - \Prob*{E_i} \Prob*{E_j}}}{\rbras*{\sum_{i = 1}^n \Prob*{E_j}}^2} \leq \frac{\pi^2n}{3(n c_E)^2} = \frac{\pi^2}{3c_E^2} \cdot \frac{1}{n}
	\end{align*}
	and therefore
	\begin{align*}
		\lim\limits_{n \rightarrow \infty} \abs*{ \frac{\sum_{i,j = 1, i \neq j}^n \sbras*{\Prob*{E_i \cap E_j} - \Prob*{E_i} \Prob*{E_j}}}{\rbras*{\sum_{i = 1}^n \Prob*{E_j}}^2} } \leq \lim\limits_{n \rightarrow \infty} \frac{\pi^2}{3c_E^2} \cdot \frac{1}{n} = 0,
	\end{align*}
	which implies \eqref{eq:mod_borel_cantelli} and Lemma \ref{lem:mod_borel_cantelli} yields
	\begin{align*}
		\Prob*{\limsup\limits_{i \rightarrow \infty} E_i} = 1,
	\end{align*}
	and thus
	\begin{align*}
		\Prob*{E_{\theta_i} \text{ for infinitly many } i \in \NN } = 1
	\end{align*}
	which completes the proof of Theorem \ref{thm:perc_vr}. The proof of \ref{thm:perc_c} is similar.
\end{proof}

\bigskip

{\bf Acknowledgement:} The authors gratefully acknowledge support
from the DFG in form of the Research Training Group 1916 ``Combinatorial
Structures in Geometry''.

\bibliographystyle{mrei}


\end{document}